\documentclass[11pt]{amsart}
\usepackage{amsmath, amssymb, epsfig}
\usepackage{amsmath}
\usepackage{amssymb}
\usepackage{bm}
\usepackage{mathrsfs}
\usepackage{color}

\newlength{\algorithmwidth}
\newcommand{\norm}[1]{\Vert #1 \Vert}

\newcommand{\normtwo}[1]{\Vert #1 \Vert_2}

\newcommand{\normone}[1]{\Vert #1 \Vert_1}

\newcommand{\gr}[1]{( #1 )}
\newcommand{\Gr}[1]{\big( #1 \big)}
\newcommand{\GR}[1]{\Big( #1 \Big)}
\newcommand{\GRg}[1]{\bigg( #1 \bigg)}

\newcommand{\C}{\mathbb C}

\newcommand{\eps}{\varepsilon}

\newcommand{\bs}[1]{#1}

\newcommand{\e}{\mathrm{e}} 

\newcommand{\cs}{compressed sensing}

\newcommand{\wrt}{with respect to}

\theoremstyle{plain}
\newtheorem{theorem}{Theorem}[section]

\newtheorem{corollary}[theorem]{Corollary}
\newtheorem{lemma}[theorem]{Lemma}

\theoremstyle{definition}
\newtheorem{definition}[theorem]{Definition}

\theoremstyle{remark}

\numberwithin{equation}{section}

\newcommand{\enorm}[1]{\norm{#1}_2}

\newcommand{\pnorm}[2]{\norm{#2}_{#1}}

\newcommand{\defby}{\overset{\mathrm{\scriptscriptstyle{def}}}{=}}

\def \C {\mathbb{C}}

\def \e {\varepsilon}
\def \eps {\varepsilon}

\def \< {\langle}
\def \> {\rangle}
\def \^ {\widehat}

\def \supp {{\rm supp}}

\newcommand{\bigO}{\mathrm{O}}

    \newcommand{\vct}{}

\begin{document}
\bibliographystyle{plain}
\title[]{Mixed Operators in Compressed Sensing}

\author{Matthew A. Herman \and Deanna Needell}
\thanks{D.N.~is with the Dept.~of Statistics, Stanford University, 390 Serra Mall, Stanford CA 94305, USA. e-mail:
\texttt{dneedell@stanford.edu}.} %
\thanks{M.H.~is with the Dept.~of Mathematics, University of California, Los Angeles, 520 Portola Plaza, Los Angeles, CA 90095, USA. e-mail:
\texttt{mattyh@math.ucla.edu}.}
\thanks{M.H. is partially supported by NSF Grant No. DMS-0811169, NSF VIGRE Grant No. DMS-0636297, and a grant from the DoD at UCLA. D.N. is partially supported by the NSF DMS EMSW21-VIGRE grant}
\begin{abstract}
Applications of compressed sensing motivate the possibility of using
different operators to encode and decode a signal of interest. Since
it is clear that the operators cannot be too different, we can view
the discrepancy between the two matrices as a perturbation. The
stability of $\ell_1$-minimization and greedy algorithms to recover
the signal in the presence of \emph{additive noise} is by now
well-known. Recently however, work has been done to analyze these
methods with noise in the measurement matrix, which generates a
\emph{multiplicative noise} term. This new framework of generalized
perturbations (i.e., both additive and multiplicative noise) extends
the prior work on stable signal recovery from incomplete and
inaccurate measurements of Cand\`es, Romberg and Tao using Basis
Pursuit (BP), and of Needell and Tropp using Compressive Sampling
Matching Pursuit (CoSaMP). We show, under reasonable assumptions,
that the stability of the reconstructed signal by both BP and CoSaMP
is limited by the noise level in the observation. Our analysis
extends easily to arbitrary greedy methods.
\end{abstract}
\maketitle

\section{Introduction}


Compressed sensing refers to the problem of realizing a sparse, or
nearly sparse, signal from a small set of linear measurements. There
are many applications of compressed sensing in engineering and the
the sciences. Examples include biomedical imaging, x-ray
crystallography, audio source separation, seismic exploration, radar
and remote sensing, telecommunications, distributed and multi-sensor
networks, machine learning, robotics and control, astronomy, surface
metrology, coded aperture imaging, biosensing of DNA, and many more.
See~\cite{CSwebpage} for an extensive list of the latest literature.

To precisely formulate the problem, we define an $s$-sparse signal
$x\in\C^d$ to be one with~$s$ or fewer non-zero components,
$$
\|x\|_0 \defby |\supp(x)| \leq s \ll d.
$$
We apply a matrix $A\in\C^{m\times d}$ to the signal and acquire
measurements 
$b = Ax$. Often, we encounter additive noise so that the
measurements become $y = b + e = Ax + e$, where~$e$ is an error or
noise term usually assumed to have bounded energy $\normtwo{e} \le
\epsilon$. The field of \cs\ has provided many recovery algorithms
for sparse and nearly sparse signals, most with strong theoretical
and numerical results.

One major approach to sparse recovery is $\ell_1$-minimization or
Basis Pursuit~\cite{DS89:Uncertainty-Principles,CT05:Decoding}. This
method simply solves an optimization problem to recover the signal
$x$,
\begin{equation}\label{L1}
\min_z \|z\|_1 \quad \text{such that} \quad \|A z - y\|_2 \leq
\epsilon.
\end{equation}
This problem can be solved using convex optimization techniques and
is thus computationally feasible. Cand\`es and Tao show
in~\cite{CT05:Decoding} that if the signal $x$ is sparse and the
measurement matrix $A$ satisfies a certain quantitative property,
then~\eqref{L1} recovers the signal $x$ exactly.

\begin{definition}
A measurement matrix $A$ satisfies the \emph{restricted isometry
property} (RIP) with parameters $(s,\delta)$ if for every $s$-sparse
vector $x$, we have
$$
(1-\delta)\|x\|_2^2 \leq \|A x\|_2^2 \leq (1+\delta)\|x\|_2^2.
$$
The parameter $\delta$ is also referred to as the \emph{restricted
isometry constant} (RIC) of matrix $A$.
\end{definition}

It is now well known that many $m \times d$ matrices (e.g., random
Gaussian, Bernoulli, and partial Fourier) satisfy the RIP with
parameters $(s, \delta)$ when $m = \bigO(s\log d)$,
see~\cite{MPJ06:Uniform,RV08:sparse} for details. It has been shown
in ~\cite{CT05:Decoding,CRT06:Stable} that if $A $ satisfies the RIP
with parameters $(3s, 0.2)$, then~\eqref{L1} recovers a signal
$x^\star$ that satisfies
\begin{equation}\label{bound}
\|x^\star - x\|_2 \leq C_0\frac{\|x-x_s\|_1}{\sqrt{s}} + C_1\epsilon
\end{equation}
where $x_s$ denotes the vector consisting of the $s$ largest
components of $x$ in magnitude.

In~\cite{Can08:Restricted-Isometry} Cand\`es sharpened this bound to
work for matrices satisfying the RIP with parameters $(2s,
\sqrt{2}-1)$, and later Foucart and Lai sharpened it to work for
$(2s, 0.4531)$ in~\cite{FL08:Sparsest}.

Although the recovery guarantees provided by $\ell_1$-minimization
are strong, it requires methods of convex optimization which,
although often quite efficient in practice, have a polynomial
runtime. For this reason, much work in compressed sensing has been
done to find faster methods. Many of these algorithms are greedy,
and compute the (support of the) signal iteratively (see e.g.,
~\cite{TG07:Signal-Recovery,BD08:Iterative,DTDS06:Sparse-Solution,FR07:Iterative,NV07:Uniform-Uncertainty,DM08:Subspace-Pursuit}).
Our analysis in this work focuses on Needell and Tropp's Compressive
Sampling Matching Pursuit (CoSaMP)~\cite{NT08:Cosamp}. CoSaMP
provides a fast runtime while also providing strong guarantees
analogous to those of $\ell_1$-minimization.

The CoSaMP algorithm can be described as follows. We use the
notation $w|_T$ and $A_T$ to denote the vector $w$ restricted to
indices given by a set $T$, and the matrix $A$ restricted to the the
columns indexed by $T$, respectively. 

\bigskip
\textsc{Compressive Sampling Matching Pursuit (CoSaMP)}

\nopagebreak

\fbox{\parbox{\algorithmwidth}{ \textsc{Input:} Measurement matrix
$A$, measurement vector $\vct{y}$, sparsity level $s$

\textsc{Output:} $s$-sparse reconstructed vector $\hat{\vct{x}} =
\vct{a}$

  \textsc{Procedure:}

\begin{description}

\item[Initialize] Set $\vct{a}^0 = \vct{0}$, $\vct{v} = \vct{y}$,
$k = 0$. Repeat the following steps and\\
increment $k$ until the halting criterion is true.

\item[Signal Proxy] Set $\vct{u} = A^* \vct{v}$,
$\Omega = \supp{ (\vct{u}_{2s}) }$ and merge the supports:\\
$T = \Omega \cup \supp{ (\vct{a}^{k-1}) }$.

\item[Signal Estimation] Using least-squares, set $\vct{w}|_{T} =
A_T^\dagger \vct{y}$ and $\vct{w}|_{T^c} = \vct{0}$.

\item[Prune] To obtain the next approximation, set $\vct{a}^{k} =
\vct{w}_s$.

\item [Sample Update] Update the current samples: $\vct{v} =
\vct{y} - A \vct{a}^{k}$.

\end{description}
}}

\bigskip


In~\cite{NT08:Cosamp} it is shown that when the measurement matrix has a small RIC 
that CoSaMP approximately recovers arbitrary signals from noisy
measurements. This is summarized by the following.

\begin{theorem}[CoSaMP~\cite{NT08:Cosamp}] \label{thm:cosamp}
Suppose that $A$ is a measurement matrix with RIC
$\delta_{4s} \leq 0.1$. Let $\vct{y} = A \vct{x} + \vct{e}$ be a
vector of samples of an arbitrary signal $x$, contaminated with
arbitrary noise. Then the algorithm CoSaMP produces an $s$-sparse
approximation $x^\sharp$ that satisfies
$$
\enorm{ x^\sharp - \vct{x} } \leq {C} \cdot  \big( \enorm{ \vct{x} -
\vct{x}_s } + \frac{1}{\sqrt{s}}\pnorm{1}{\vct{x} - \vct{x}_s} +
\enorm{ \vct{e} } \big).
$$
\end{theorem}

\section{Mixed Operators}

Applying the theories of compressed sensing to real-world problems
raises the following question: what happens when the operator used
to \emph{encode} the signal is \emph{different} from the operator
used to \emph{decode} the measurements? In many of the applications
mentioned in Section~1
the sensing, or measurement, matrix $\bs{A}$ actually represents a
system which the signal passes through. In other scenarios $\bs{A}$
represents some other physical phenomenon. For example, in screening
for genetic disorders~\cite{EGBHM10:Genotyping}, the standard deviation of
error in the sensing matrix is about $3\%$. This error is due to
human handling when pipetting the DNA
samples~\cite{E10:Genotyping_PersonalCommunication}. Whatever the
setting may be, it is often the case that the true nature of this
system is not known exactly. When this happens the system behavior
is (perhaps unknowingly) approximated, or assumed to be represented,
by a different matrix~$\bs{\Phi}$.

It is clear that the encoding and decoding operators, $\bs{A}$ and
$\bs{\Phi}$, cannot be too different, but until recently there has
been no analysis of the effect this difference has on reconstruction
error.  In particular, the perturbations in the sensing matrices create a multiplicative noise term of the form $(A-\Phi)x$.  Herman and Strohmer first showed in~\cite{HS10:General} that
a noisy measurement matrix can be successfully used to recover a
signal using $\ell_1$-minimization.
The natural question is whether this extends to the case of greedy
algorithms.
In this work, we consider the case of mixed operators in CoSaMP, and
our results naturally apply to other greedy algorithms.

In this analysis we will require examination of submatrices of
certain matrices. To that end, we define
${\normtwo{\bs{A}}^{\gr{s}}}$ to be the largest spectral norm over
all $s$-column submatrices of
~$A$. Let
\begin{equation}\label{pert_constants}
\eps^{\gr{s}}_{\bs{A}} \defby \frac{\normtwo{\bs{A-\Phi}}^{\gr{s}}}
{\normtwo{\bs{A}}^{\gr{s}}} \quad \text{ and } \quad
\kappa^{\gr{s}}_{\bs{A}} \defby
\frac{\sqrt{1+\delta_s}}{\sqrt{1-\delta_{s}}}.
\end{equation}
The first quantity is the relative perturbation of $s$-column
submatrices of $A$ \wrt\ to the spectral norm, and the second one
bounds ratio of the extremal singular values of all $s$-column
submatrices of $A$ (see~\cite{HS10:General} for more details). We
also need a measure of how ``close'' a signal $x$ is to a sparse
signal, and therefore define
\begin{equation}\label{alphabeta}
\alpha_s \defby \frac{\enorm{x - x_s}}{\enorm{x_s}} \quad \text{ and
} \quad \beta_s \defby \frac{\pnorm{1}{x-x_s}}{\sqrt{s}\enorm{x_s}}.
\end{equation}

\subsection{Mixed Operators in $\ell_1$-minimization}

The work in~\cite{HS10:General} extended the previous results in
$\ell_1$-minimization by generalizing the error term~$\epsilon$
which only accounted for additive noise. The new framework considers
a \emph{total noise term~}$\eps_{\bs{A},s,\bs{b}}$ which allows for
both multiplicative and additive noise. Theorem~\ref{thm:HSthm}
below shows that the reconstruction error using
$\ell_1$-minimization is limited by this noise level. With regard to
noise in operator~$A$ we see that the stability of the solution is a
linear function of relative
perturbations~$\eps_{\bs{A}},\eps_{\bs{A}}^{\gr{s}}$.

\begin{theorem}[Adapted from~\cite{HS10:General}, Thm.~2]\label{thm:HSthm}
Let~$x$ be an arbitrary signal with measurements $b = Ax$, corrupted
with noise to form $y = Ax + e$.
Assume the RIC for matrix $\bs{A}$ satisfies
\begin{equation} \label{cond:Main_thm_constraint_1}
\delta_{2s} \;<\; \frac{\sqrt{2}}
{\GR{1+\eps_{\bs{A}}^{\gr{2s}}}^{2}} \,-\, 1
\end{equation} and that general signal $\bs{x}$ satisfies
\begin{equation} \label{cond:Main_thm_constraint_2}
\alpha_s + \beta_s \;<\; \frac{1}{\kappa_{\bs{A}}^{\gr{s}}}.
\end{equation}
Set the total noise parameter
\begin{equation} \label{eq:BP_absolute_error_constraint}
\eps_{\bs{A},s,\bs{b}} := \GRg{
\frac{\eps_{\bs{A}}^{\gr{s}}\kappa_{\bs{A}}^{\gr{s}} +\,
\eps_{\bs{A}}\gamma_{\bs{A}}\alpha_s}
{1-\kappa_{\bs{A}}^{\gr{s}}\!\Gr{\alpha_s + \beta_s}} \,+\,
\eps_{\bs{b}}} \normtwo{\bs{b}}
\end{equation}
where the relative perturbations $\eps_{\bs{A}} =
\frac{\normtwo{\bs{A-\Phi}}}{\normtwo{\bs{A}}}$, $\eps_{\bs{b}}  =
\frac{\normtwo{\bs{e}}}{\normtwo{\bs{b}}}$, and $\gamma_{\bs{A}} =
\frac{\normtwo{\bs{A}}}{\sqrt{1-\delta_{s}}}$. Then the solution $\bs{z^\star}$ to the BP
problem~(\ref{L1}) with $\epsilon$ set to $\eps_{\bs{A},s,\bs{b}}$, and using the decoding matrix $\bs{\Phi}$ (instead of
$\bs{A}$) obeys
\begin{equation} \label{eq:Pert_error_soln_CS_l1}
{\normtwo{\bs{z^\star} - \bs{x}}} \;\le\;
\frac{C_0}{\sqrt{s}}\:\!\normone{\bs{x}-\bs{x}_s} \,+\,
C_1\:\!\eps_{\bs{A},s,\bs{b}} \\
\end{equation}
for some well-behaved constants $C_0, C_1$.
\end{theorem}

\subsection{Mixed Operators in CoSaMP}
We now turn to the case of mixed operators in greedy algorithms, and
in particular CoSaMP.  We will see that a result analogous to that
of $\ell_1$-minimization can be obtained in this case as well.
Similar to condition~(\ref{cond:Main_thm_constraint_2}) above, we
will again need for the signal to be well approximated by a sparse
signal. To that end,
we require that
\begin{equation}\label{thecond}
\alpha_s + \beta_s \,\leq\, \frac{1}{2\,\kappa_{\bs{A}}^{\gr{s}}}
\end{equation}
where $\alpha_s$ and $\beta_s$ are defined in~\eqref{alphabeta}.
Theorem~\ref{thm:thmarb} below shows that under this assumption, the
reconstruction error in CoSaMP is again limited by the tail of the
signal and the observation noise.

\begin{theorem}\label{thm:thmarb}
Let $A$ be a measurement matrix with RIC 
\begin{equation}\label{the2ndcond}
\delta_{4s} \leq \frac{1.1}{(1+\eps^{\gr{4s}}_{\bs{A}})^2}-1.
\end{equation}
Let~$x$ be an arbitrary signal with measurements $b = Ax$, corrupted
with noise to form $y = Ax + e$. Let $x^\sharp$ be the
reconstruction from CoSaMP using decoding matrix $\Phi$ (instead of
$\bs{A}$) on measurements $y$. Then if~\eqref{thecond} is satisfied,
the estimation satisfies
$$
\|x^\sharp - x\|_2 \leq C \cdot \left( \enorm{x-x_s} +
\frac{1}{\sqrt{s}}\pnorm{1}{x - x_s} + (\e\alpha_s +
\e^{(s)})\enorm{b}  + \enorm{e}  \right)
$$
where $\e = \|A-\Phi\|_2$ and $\e^{(s)} = \enorm{A - \Phi}^{(s)}$.
\end{theorem}

Applying Theorem~\ref{thm:thmarb} to the sparse case, we immediately
have the following corollary.

\begin{corollary}\label{mainthm}
Let $A$ be a measurement matrix with RIC 
$\delta_{4s} \leq \frac{1.1}{(1+\eps^{\gr{4s}}_{\bs{A}})^2}-1$.
Let~$x$ be an $s$-sparse signal with noisy measurements $y = b + e =
Ax + e$. Let $x^\sharp$ be the reconstruction from CoSaMP using
decoding matrix $\Phi$ (instead of $\bs{A}$). Then the estimation
satisfies
$$
\|x^\sharp - x\|_2 \leq C \cdot \left( \e^{(s)}\|b\|_2 + \|e\|_2
\right)
$$
where $\e^{(s)} = \enorm{A - \Phi}^{(s)}$.
\end{corollary}

We now analyze the case of mixed operators and prove our main
result, Theorem~\ref{thm:thmarb}. We will first utilize a result
from~\cite{HS10:General} which states that matrices which are
``close'' to each other also have similar RICs.

\begin{lemma}[RIP for $\Phi$~\cite{HS10:General}] \label{thm:RIP_perturbed}
For any $s=1,2,\ldots$, assume and fix the RIC~$\delta_s$ associated
with $\bs{A}$, and the relative
perturbation~$\eps^{\gr{s}}_{\bs{A}}$ associated with $\bs{A-\Phi}$
as defined in~(\ref{pert_constants}). Then the RIC constant
$\hat{\delta}_s$ for matrix $\bs{\Phi}$ satisfies
\begin{equation} \label{eq:pert_RIC_bound}
\hat{\delta}_s \,\leq\, \Gr{1+\delta_s}\GR{1+\eps^{\gr{s}}_{\bs{A}}}^2
- 1. \end{equation}
\end{lemma}

We now prove our main result, Theorem~\ref{thm:thmarb}
\begin{proof}[Proof of Theorem~\ref{thm:thmarb}]
Lemma~\ref{thm:RIP_perturbed} applied to the case where $\delta_{4s}
\leq \frac{1.1}{(1+\eps^{\gr{4s}}_{\bs{A}})^2}-1$ implies that the
matrix $\Phi$ has an RIC 
that satisfies $\hat{\delta}_{4s} \leq 0.1$. We
can then apply Theorem~\ref{thm:cosamp} with measurements $y = \Phi
x + (A - \Phi)x + e$. This implies that the reconstruction
$x^\sharp$ satisfies
\begin{equation}\label{start}
\enorm{x^\sharp - \vct{x}} \leq {C} \cdot  \left(\enorm{x - x_s} +
\frac{1}{\sqrt{s}}\pnorm{1}{x-x_s} + \enorm{ (A - \Phi)x} +
\enorm{e}\right).
\end{equation}

As seen in Proposition 3.5 of~\cite{NT08:Cosamp}, the RIP 
implies that for an arbitrary signal $x$,
$$
\enorm{Ax} \leq \sqrt{1 + \delta_s}
    \left( \enorm{ \vct{x} } + \frac{1}{\sqrt{s}}
        \pnorm{1}{\vct{x}} \right).
$$
As shown in~\cite{HS10:General}, this and the RIP 
imply that
$$
\enorm{Ax} \geq \sqrt{1 - \delta_s}\enorm{x_s} - \sqrt{1 +
\delta_s}\big(\enorm{x-x_s} +
\frac{1}{\sqrt{s}}\pnorm{1}{x-x_s}\big).
$$
We then have that
$$
\enorm{(A-\Phi)x} \leq \frac{\enorm{A-\Phi}\enorm{x-x_s} +
\enorm{A-\Phi}^{(s)}\enorm{x_s}}{\sqrt{1 - \delta_s}\enorm{x_s} -
\sqrt{1 + \delta_s}(\enorm{x-x_s} +
\frac{1}{\sqrt{s}}\pnorm{1}{x-x_s})}\enorm{Ax}.
$$
Condition~\eqref{thecond} then gives us
\begin{align*} \enorm{(A-\Phi)x} &\leq \frac{\enorm{A-\Phi}
\enorm{x-x_s} + \enorm{A-\Phi}^{(s)}\enorm{x_s}}
{\frac{1}{2}\sqrt{1-\delta_s}\enorm{x_s}}\enorm{Ax}\\
&= \left(\frac{2\enorm{A-\Phi}}{\sqrt{1-\delta_s}}
+ \frac{2\enorm{A-\Phi}^{(s)}}{\sqrt{1-\delta_s}} \right)\enorm{Ax}.
\end{align*}
Applying the inequality $\delta_s \leq \delta_{4s} \leq 0.1$ yields
$$
\enorm{(A-\Phi)x} \leq C'\big(\enorm{A-\Phi}\alpha_s +
\enorm{A-\Phi}^{(s)}\big)\enorm{Ax}.
$$
Combined with~\eqref{start}, this completes the claim.
\end{proof}

\section{Discussion}


One should of course make sure that the requirements imposed by
Theorems~\ref{thm:HSthm} and~\ref{thm:thmarb} are reasonable and
make sense. For instance, in Theorem~\ref{thm:HSthm}, to ensure that
the RIC $\delta_{2s} \ge 0$ we can set the left-hand side of
condition~(\ref{cond:Main_thm_constraint_1}) to zero. Rearranging, this 
requires that $\eps_{\bs{A}}^{(2s)} < \sqrt[4]{2}-1$, which
addresses the question ``how \emph{dissimilar} can $A$ and $\Phi$
be?" Loosely phrased, the answer is that the spectral norm of
$2s$-column submatrices of $\Phi$ cannot deviate by more than about
$19\%$ of spectral norm of $2s$-column submatrices of $A$. The
corresponding condition~(\ref{the2ndcond}) in
Theorem~\ref{thm:thmarb} requires that $\eps_{\bs{A}}^{(4s)} \le
\sqrt{1.1}-1$, which translates to an approximate $5\%$
dissimilarity between $A$ and $\Phi$. The second condition,
(\ref{cond:Main_thm_constraint_2}), in Theorem~\ref{thm:HSthm} is
discussed in~\cite{HS10:General}, and essentially requires that the
signal be well approximated by a sparse signal. This is, of course,
a standard assumption in compressed sensing.  The same argument
holds for condition~(\ref{thecond}) in Theorem~\ref{thm:thmarb}.

In conclusion, real-world applications often utilize different
operators (perhaps unknowingly) to encode and decode a signal.  The perturbation 
of the sensing matrix creates \textit{multiplicative} noise in the system.  This 
type of noise is \textit{fundamentally} different than simple additive noise.  For example, to overcome 
a poor signal-to-noise ratio (SNR)
due to additive noise, one would typically increase the strength of the signal.  However, 
if the noise is multiplicative this will not improve the situation, and in fact will actually cause the error to grow.  
Thus the impact on reconstruction from the error in the sensing matrices needs to be analyzed.  Our
Theorems~\ref{thm:HSthm} and~\ref{thm:thmarb} do just that.  They show the effect
of using mixed operators to recover a signal in compressed sensing:
the stability of the recovered signal is a linear function of the
operator perturbations defined above.  This work confirms that this
is the case both for $\ell_1$-minimization and CoSaMP. These results
can easily be extended to other greedy algorithms as well.

 \subsection*{Acknowledgment}
We would like to thank Thomas Strohmer for many thoughtful discussions and his invaluable guidance.

\bibliography{v4}
\end{document}